\documentclass[preprint,12pt]{elsarticle}
\usepackage{amsmath,amssymb,amsthm}
\usepackage{mathrsfs}
\usepackage{geometry}
\theoremstyle{definition}   % Upright font for all the following environments
\newtheorem{theorem}{Theorem}[section]
\newtheorem{lemma}[theorem]{Lemma}
\newtheorem{definition}[theorem]{Definition}
\newtheorem{remark}[theorem]{Remark}
\newtheorem{corollary}[theorem]{Corollary}

\newtheorem{assumption}[theorem]{Assumption}

\numberwithin{equation}{section}

\def\pt{\partial}

\def\ra{\rightarrow}
\def\bs{\boldsymbol}

\def\s{\subseteq}

\def\sps{\supseteq}

\def\ol{\overline}

\def\ba{\bigcap}
\def\bu{\bigcup}
\def\ms{\mathscr}
\def\vp{\varphi}
\def\lg{\langle}
\def\llg{\left\langle}

\def\rg{\rangle}
\def\rrg{\right\rangle}

\def\mf{\mathfrak}

\def\pt{\partial}

\def\Om{\Omega}
\def\la{\lambda}

\def\be{\beta}

\def\ga{\gamma}

\def\si{\sigma}

\def\Ga{\Gamma}
\def\La{\Lambda}

\def\ts{\times}

\def\iy{\infty}

\def\f{\frac}

\def\se{\setminus}

\def\Lra{\Leftrightarrow}
\def\Ra{\Rightarrow}

\def\df{\mathrm d}

\def\wt{\widetilde}
\def\wh{\widehat}

\def\esssup{\operatorname*{ess\ \! sup}}

\def\mcH{\mathcal{H}}

\def\mcL{\mathcal{L}}

\def\mcS{\mathcal{S}}

\def\mcO{\mathcal{O}}
\def\mcA{\mathcal{A}}

\def\mcD{\mathcal{D}}

	\DeclareMathOperator{\Div}{div}

	\DeclareMathOperator{\dist}{dist}

	\DeclareMathOperator{\supp}{{supp}}

	\newcommand{\R}{\mathbb R}
	\newcommand{\N}{\mathbb N}

\usepackage{hyperref}

\begin{document}

\begin{frontmatter}
\title{Non-homogeneous boundary value problems for second-order degenerate hyperbolic equations and their application}

\author[aff1]{Donghui Yang}
\ead{donghyang@outlook.com}
\author[aff2]{Jie Zhong\corref{cor1}}
\ead{jiezhongmath@gmail.com}
\cortext[cor1]{Corresponding author.}

\address[aff1]{School of Mathematics and Statistics, Central South University, Changsha 410075, China}
\address[aff2]{Department of Mathematics, California State University Los Angeles, Los Angeles, CA, USA}

\begin{abstract}
We study second-order hyperbolic equations with degenerate elliptic operators and non-homogeneous Dirichlet boundary inputs. We establish existence and regularity of weak solutions in weighted Sobolev spaces under mild assumptions on the degenerate weight. A Dirichlet map is constructed for the degenerate elliptic operator, leading to a solution theory that extends classical approaches to the degenerate setting. In particular, we derive energy estimates and well-posedness for boundary inputs of low regularity (in appropriate trace spaces), even though the classical Dirichlet-to-Neumann framework is not directly applicable in the degenerate setting. As an application, we prove an approximate controllability criterion, which generalizes the Hilbert Uniqueness Method to degenerate wave equations. Our framework accommodates higher-dimensional degenerate waves, non-homogeneous boundary conditions, and weighted functional analysis. We also illustrate how our criterion connects to higher-dimensional Grushin equations and waves with single-point degeneracy, and we highlight the remaining unique continuation/observability issue as an open problem.
\end{abstract}


\end{frontmatter}

\section{Introduction}

The controllability and stabilization of wave-type (second-order hyperbolic) equations have been extensively studied in the literature. Foundational works by, among others, G.~Chen \cite{Chen1, Chen2}, H.O.~Fattorini \cite{Fattorini1}, I.~Lasiecka and R.~Triggiani \cite{Lasiecka,Lasiecka1,Lasiecka2,Lasiecka3}, M.~Slemrod \cite{Slemrod}, P.-F.~Yao \cite{Yao}, and E.~Zuazua \cite{Zuazua} have established many of the classical results for uniformly hyperbolic equations with various types of boundary or internal controls. In particular, Lasiecka and Triggiani developed a rigorous framework for non-homogeneous boundary inputs in second-order hyperbolic equations using semigroup and cosine operator methods (see \cite{Lasiecka1,Lasiecka2}). For example, Lasiecka, Lions, and Triggiani \cite{Lasiecka1} analyzed wave equations with $L^2(0,T;L^2(\Gamma))$ Dirichlet boundary controls, introducing the concept of a Dirichlet map. However, these classical methods assume uniformly elliptic spatial operators. Hyperbolic equations with variable or degenerate coefficients pose additional challenges, often requiring advanced techniques from differential geometry (see \cite{Lasiecka3,Yao}).

In recent years, \emph{degenerate hyperbolic equations} have attracted growing attention. By ``degenerate'', we mean wave equations whose spatial differential operator vanishes or degenerates in part of the domain, causing the wave speed to vanish locally. Several papers have addressed such problems. Most of these studies, however, focus on one-dimensional cases (see \cite{Cannarsa1,Fragnelli,Gueye,Gao}). In higher dimensions, results are much more limited. Notably, Yang et al.\ studied multi-dimensional degenerate waves: in a recent work \cite{Yang}, Yang, Wu, Guo, and Chai obtained exact controllability for a two-dimensional Grushin-type wave equation, which degenerates along a boundary line. A further preprint \cite{Yang1} extends controllability results to an $N$-dimensional wave with an internal single-point degeneracy. Although related controllability/observability results exist for certain degenerate models (see \cite{Yang,Yang1}), transferring them to the present boundary Dirichlet-control setting and our weighted framework may require additional unique continuation/trace analysis, which we do not attempt here. These higher-dimensional cases exhibit fundamentally different behavior from the one-dimensional setting and require new analytical tools.

Degenerate hyperbolic equations in multiple space dimensions present significant analytical difficulties. To establish controllability or stabilization results, one must first prove well-posedness and regularity of solutions---a highly nontrivial task. The main challenges lie in developing appropriate weighted Sobolev spaces to handle the degeneracy and in constructing a Dirichlet map or solution operator for non-homogeneous boundary data when the elliptic operator is degenerate. These aspects sharply distinguish degenerate hyperbolic equations from uniformly elliptic ones. In the classical theory, one often uses the Dirichlet-to-Neumann map or builds solutions via elliptic lifting of boundary traces. In the degenerate case, however, such an approach is not straightforward: the standard Dirichlet map may fail to exist or be bounded under degeneracy conditions. For instance, the results in \cite{Lasiecka1,Lasiecka2} of Lasiecka and Triggiani on hyperbolic equations with $L^2$ boundary inputs cannot be directly extended, because the underlying elliptic operator is no longer invertible in the usual sense.

In this paper, we overcome these difficulties by developing a new framework for degenerate hyperbolic equations with boundary control. To this end, we introduce the concept of weight functions. We use the notion of $A_p$ weights. When $p=1$, a locally integrable nonnegative function $w(\cdot)$ is an $A_1$ weight if there exists a constant $c=c(w,1)>0$ such that for all cubes $K\subset\mathbb{R}^N$,
\[
\frac{1}{|K|}\int_K w(x)\, dx \le c \cdot \operatorname*{ess\,inf}_{x\in K} w.
\]
For $p\in(1,\infty)$, a locally integrable nonnegative function $w(\cdot)$ is called an $A_p$ weight if there exists a constant $c>0$ such that for all cubes $K\subset\mathbb{R}^N$,
\begin{equation}
\Bigg(\frac{1}{|K|}\int_K w(x)\, dx\Bigg)
\Bigg(\frac{1}{|K|}\int_K w(x)^{-\frac{1}{p-1}}\, dx\Bigg)^{p-1}
\le c.
\label{01.18.2}
\end{equation}
The infimum of such constants $c=c(w,p)$ is referred to as the $A_p (1\le p \le \infty)$ constant of $w(\cdot)$.

We focus on a prototypical second-order degenerate wave equation with Dirichlet boundary input:
\begin{equation}
\begin{cases}
\partial_{tt}y(x,t) - \operatorname{div}\!\big(w(x)\nabla y(x,t)\big) = 0, & (x,t)\in Q:=\Omega\times(0,T),\\
y(x,0) = y_0(x),\quad \partial_t y(x,0) = y_1(x), & x\in \Omega,\\
y(\xi, t) = u(\xi, t), & (\xi,t)\in \Sigma:=\Gamma\times(0,T),
\end{cases}
\label{07.08.01}
\end{equation}
where $\Omega\subset \mathbb{R}^N$ ($N\ge 1$) is a smooth bounded domain with boundary $\Gamma = \partial\Omega$, and $T>0$ is a fixed time horizon. The coefficient $w(x)$ is a weight function that may vanish or degenerate in parts of $\Omega$; specifically, we assume $w\in C(\overline{\Omega})$ is an $A_2$ weight with $w(x)>0$ on $\Gamma$. The function $u(t,\xi)$ is the imposed Dirichlet boundary control along $\Sigma = \Gamma \times (0,T)$. The initial data are $y_0 \in H^1_0(\Omega;w)$ and $y_1 \in L^2(\Omega)$. We denote
\begin{equation}
Ay = - \operatorname{div}(w\nabla y).
\label{07.08.02}
\end{equation}

Our primary goals are twofold. First, we establish the existence, uniqueness, and regularity of solutions $y(x,t)$ to \eqref{07.08.01} in suitable function spaces, even with non-zero boundary input $u(t,\xi)$. We develop this theory in Section~2 by introducing weighted Sobolev spaces $H^1(\Omega;w)$ and $H^2(\Omega;w)$ adapted to the degenerate operator $A = -\operatorname{div}(w\nabla\cdot)$. Under structural assumptions on $w$ and $\Omega$, we prove that for any admissible boundary control $u(\xi, t)$ and source term $f(x,t)$, the equation \eqref{07.08.01} possesses a unique weak solution with sharp energy estimates. In particular, we show that $y \in C([0,T];H^1(\Omega;w)) \cap C^1([0,T];L^2(\Omega))$ when $u\in H^1(\Sigma)$, and we derive a priori bounds on $y$ in terms of the norms of $u$, $f$, and the initial data. These results generalize the well-posedness theory for non-degenerate boundary-controlled waves (see \cite{Lions}) to the degenerate case.

A crucial step in our analysis is the construction of a Dirichlet map for the degenerate operator $A$. In Section~2.2, we introduce an operator $D$ assigning to each boundary function $\psi$ a function $y=D\psi$ in $\Omega$ such that $Ay=0$ in $\Omega$ and $y|_{\Gamma}=\psi$. This map allows us to lift boundary data into the interior. We prove that $D$ is well-defined and bounded between appropriate spaces, and we use it to decompose solutions of \eqref{07.08.01} into homogeneous and non-homogeneous parts. This leads to energy estimates for the homogeneous component and regularity results for the boundary traces.

Second, as an application of the above theory, we investigate approximate boundary controllability of the degenerate wave equation \eqref{07.08.01}. Roughly speaking, approximate controllability means that, starting from zero initial data, one can drive the solution arbitrarily close to any desired target state in the energy space using boundary controls. Using the well-posedness and trace results, we formulate the Hilbert Uniqueness Method (HUM) for our degenerate system. We show that controllability is equivalent to an observability condition: the only solution of the adjoint equation with vanishing boundary normal derivative is the trivial one. This criterion generalizes the HUM framework to degenerate wave equations. We also illustrate the connection to important examples such as multi-dimensional Grushin equations and waves with single-point degeneracy, and we leave the verification of the corresponding unique continuation/observability property as an open problem.

The paper is organized as follows. In Section~2, we present the functional setting, define weak and mild solutions, and derive energy estimates for both homogeneous and non-homogeneous boundary problems. In Section~3, we establish approximate controllability and discuss illustrative examples.

\section{Solution space and well-posedness theory}
% Insert detailed theorems, lemmas, proofs, and equations here from your polished draft
In this section, we develop the functional framework and prove the existence, uniqueness, and regularity of solutions to the degenerate wave equation. We begin by defining the weighted Sobolev spaces appropriate for our degenerate operator, and we state the assumptions on $w(x)$ that will ensure a well-behaved theory. Then we handle the homogeneous boundary condition case (Dirichlet $y=0$ on $\Gamma$) and the non-homogeneous case ($y=u$ on $\Gamma$) in turn. Throughout this section, $\Omega\subset \mathbb{R}^N$ is our fixed bounded domain with smooth boundary $\Gamma$, and $w:\Omega\to [0,\infty)$ is the weight function in the elliptic operator $A = -\operatorname{div}(w\nabla\cdot)$.

\subsection{Weighted Sobolev Spaces and Weak Solutions (Homogeneous Dirichlet Case)}

We fix a weight $w$ on $\Omega$ and assume $w\in A_2$. This is the minimal hypothesis we use to obtain the weighted Poincar\'e inequality and compact embeddings that underlie the energy estimates in Theorem~\ref{thm:2.12}. We now introduce the weighted Sobolev spaces and the operator domain needed for the homogeneous Dirichlet problem.

Define
\begin{equation*}
	\mathcal{H}^1(\Om;w)=\left\{y\in L^2(\Om)\colon \int_\Om |\nabla y|^2 w\df x<\iy \right\},
\end{equation*}
equipped with the inner product
\begin{equation*}
	(y,z)_{\mcH^1(\Om;w)}=\int_\Om (\nabla y\cdot \nabla z) w\df x+\int_\Om yz\df x,
\end{equation*}
and norm $\|y\|_{\mcH^1(\Om;w)}=\left[(y,y)_{\mcH^1(\Om;w)}\right]^\f{1}{2}$.

Let $H^1(\Om;w)$ be the closure of $C^\iy(\ol\Om)$ in $(\mcH^1(\Om;w), (\cdot,\cdot)_{\mcH^1(\Om;w)})$. Let $H_0^1(\Om;w)$ be the closure of $C_0^\iy(\Om)$ in $(H^1(\Om;w), (\cdot,\cdot)_{H^1(\Om;w)})$; its elements have zero trace on $\Ga$. We denote by $H^{-1}(\Om;w)$ the dual of $H_0^1(\Om;w)$ with pivot $L^2(\Om)$.

Define also
\begin{equation*}
	H^2(\Om;w)=\left\{y\in H^1(\Om;w)\colon A y\in L^2(\Om)\right\},
\end{equation*}
equipped with
\begin{equation*}
	(y,z)_{H^2(\Om;w)}=\int_\Om [\Div (w\nabla y)][\Div(w\nabla z)]\df x+\int_\Om (\nabla y\cdot\nabla z)w\df x+\int_\Om yz\df x,
\end{equation*}
and norm $\|y\|_{H^2(\Om;w)}=\left[(y,y)_{H^2(\Om;w)}\right]^\f{1}{2}$.

\begin{lemma}\label{08.13.L1}
  The spaces $(H^1(\Om;w),(\cdot,\cdot))_{H^1(\Om;w)}$, $(H_0^1(\Om;w),(\cdot,\cdot)_{H^1(\Om;w)})$, and \\
  $(H^2(\Om;w), (\cdot,\cdot)_{H^2(\Om;w)})$ are Hilbert spaces.
\end{lemma}

\begin{proof}
	This is a standard conclusion in Weighted Sobolev spaces. See \cite{GC,Heinonen,Trudinger}.
\end{proof}

\begin{assumption}\label{06.27.A1}\label{ass:2.2}
In addition to $w\in A_2$, we assume:
\begin{enumerate}
  \item (\textit{Non-degeneracy near the boundary}) There exist constants $\Lambda>0$ and $\beta>0$ such that
	\begin{equation}\label{08.13.1}
		w\in C^1(\ol{\mcO(\Ga;2\be)}) \mbox{ and } w\geq \La \mbox{ on } \ol{\mcO(\Ga;2\beta)},
	\end{equation}
	where $\mcO(\Ga;\beta)=\left\{x\in\Om\colon \dist(x,\Ga)<\beta\right\}$ and $\dist(x,\Ga)=\inf_{z\in \Ga}|x-z|$.
  \item (\textit{Weighted Poincar\'e inequality}) There exists $C>0$ such that for all $y\in H^1_0(\Omega;w)$,
\[
\int_{\Omega} |y(x)|^2\, dx \le C_P \int_{\Omega} w(x)|\nabla y(x)|^2\, dx.
\]

  \item (\textit{Compact embedding}) $H^1_0(\Omega;w)$ is compactly embedded in $L^2(\Omega)$.

  \item (\textit{Initial trace compatibility}) For well-posedness of boundary input problems, we assume $u(0,\xi) = y_0(\xi) = 0$ for $\xi\in \Gamma$.
\end{enumerate}
\end{assumption}

\begin{remark}
Assumption~\ref{06.27.A1}(1) enforces $w\ge \Lambda>0$ in a neighborhood of $\Gamma$, so the degeneracy is confined to the interior of $\Omega$ and the boundary remains non-degenerate. This ensures the standard boundary trace and conormal derivative tools needed for the Dirichlet map, lifting, and boundary control formulation. The case of boundary-degenerate weights requires different techniques and is not treated here.
\end{remark}

\begin{remark}\label{07.08.R1}\
\begin{enumerate}
\item Since $w>0$ on $\Gamma$, the weight is non-degenerate on the boundary. Thus, the assumption $w\geq \Lambda>0$ on $\overline{\mathcal{O}(\Gamma;2\beta)}$ in \eqref{08.13.1} of Assumption~\ref{06.27.A1} is reasonable, because we already assume $w\in C(\overline{\Omega})$ and $w>0$ on $\Gamma$. From this, it follows that $z\in H^1(\mathcal{O}(\Gamma;\beta))$ for all $z\in H^1_0(\Omega;w)$. Hence, the trace $z|_{\Gamma}$ of $z\in H^1(\Omega;w)$ coincides with the classical Sobolev trace. We will use the first part of \eqref{08.13.1} to obtain the last conclusion in Lemma~\ref{07.08.L1'}.
\item By Assumption~\ref{06.27.A1}, the seminorm
\[
\|y\|_{H^1(\Omega;w)}=\left(\int_\Omega |\nabla y|^2 w\, dx\right)^{1/2}
\]
is equivalent to the standard norm on $H^1_0(\Omega;w)$. Here and in what follows, we adopt this energy norm on $H^1_0(\Omega;w)$.

Note also that Poincaré inequalities are available for many degenerate equations. For example, the classical Grushin operator on a cube satisfies such an inequality. In general, Poincaré-type inequalities can often be derived from Hardy-type inequalities for degenerate equations.
\item There are many degenerate partial differential operators $A$ (or $\mathcal{A}$, see \eqref{07.08.02}) that satisfy Assumption~\ref{06.27.A1}. Examples include: the Grushin operator; the one-dimensional operator $\partial_x\!\big(|x|^\alpha \partial_x u\big)$ for $x\in (-1,1)$ (or $x\in (0,1)$) with $\alpha\in (0,1)$ (see \cite{Cannarsa1}); and the operator $\operatorname{div}(|x|^\alpha \nabla u)$ for $x\in \Omega\subset\mathbb{R}^N$ with $0\in \Omega$ and $\alpha\in (0,2)$ (see \cite{Yang1}).
\item Condition (4) in Assumption~\ref{06.27.A1} is a compatibility condition for the equation \eqref{07.08.01}. See, for example, Theorem~1.1 in \cite{Fattorini1}, p.~352, and also \cite{Lasiecka,Lasiecka1,Lasiecka2}.
\end{enumerate}
\end{remark}

We set $D(A) = H^2(\Omega;w)\cap H^1_0(\Omega;w)$.

\begin{definition}[Conormal derivative]\label{def:conormal}
For $A=-\Div(w\nabla\cdot)$ and $y$ with a trace on $\Gamma$, we define the conormal derivative
\[
\f{\pt y}{\pt \nu_\mcA}:=w\nabla y\cdot \nu \quad \text{on }\Gamma
\]
in the weak trace sense; namely, for $y\in H^1(\Omega;w)$ with $Ay\in L^2(\Omega)$, the boundary functional in Green's identity satisfies
\[
\int_\Omega (Ay)v\,dx+\int_\Omega w\nabla y\cdot\nabla v\,dx=\left\langle \f{\pt y}{\pt \nu_\mcA},\, v|_\Gamma\right\rangle_{H^{-1/2},H^{1/2}}
\quad \text{for all }v\in H^1(\Omega;w).
\]
\end{definition}

Under Assumption~\ref{ass:2.2}, $A:D(A)\to L^2(\Omega)$ is a positive self-adjoint operator with discrete spectrum
\[
0 < \lambda_1 \le \lambda_2 \le \cdots, \quad \lambda_n\to\infty,
\]
and there exists an orthonormal basis $\{\Phi_n\}$ of $L^2(\Omega)$ consisting of eigenfunctions of $A$. Moreover, $\{\Phi_n\}$ is orthogonal in $H^1_0(\Omega;w)$, with
\[
(\Phi_n,\Phi_m)_{H^1(\Omega;w)} = \lambda_n\delta_{nm}.
\]

\medskip

We now consider the homogeneous Dirichlet problem for the degenerate wave equation,
\begin{equation}
\begin{cases}
\partial_{tt}y + Ay = f(t,x), & (x,t)\in Q=\Omega\times(0,T),\\
y(0,x) = y_0(x),\quad \partial_t y(0,x) = y_1(x), & x\in \Omega,\\
y(t,\xi) = 0, & (\xi,t)\in \Sigma=\Gamma\times(0,T),
\end{cases}
\label{09.17.1}
\end{equation}
with data $f\in L^1(0,T;L^2(\Omega))$, $y_0\in H^1_0(\Omega;w)$, and $y_1\in L^2(\Omega)$. The definition below is the notion of solution used in Theorem~\ref{thm:2.12}.

\begin{definition}[Weak solution]\label{def:weak}
A function $y:[0,T]\to H^1_0(\Omega;w)$ is a \emph{weak solution} of \eqref{09.17.1} if
\[
y \in L^\infty(0,T;H^1_0(\Omega;w)), \quad \partial_t y \in L^\infty(0,T;L^2(\Omega)), \quad \partial_{tt}y\in L^1(0,T;H^{-1}(\Omega;w)),
\]
$y(0)=y_0$, $\partial_t y(0)=y_1$, and for all $v\in H^1_0(\Omega;w)$, almost every $t\in [0,T]$,
\[
\langle \partial_{tt}y,v\rangle_{H^{-1},H^1_0} + \int_\Omega w(x)\nabla y\cdot\nabla v\, dx = \int_\Omega f v\, dx.
\]
\end{definition}

\begin{theorem}[Existence and regularity]\label{thm:2.12}
Under Assumption~\ref{ass:2.2}:
\begin{enumerate}
\item For $y_0\in H^1_0(\Omega;w)$, $y_1\in L^2(\Omega)$, and $f\in L^1(0,T;L^2(\Omega))$, there exists a unique weak solution $y$ of \eqref{09.17.1} such that
\[
\esssup_{0\le t\le T} \big(\|y(t)\|_{H^1_0(\Omega;w)} + \|\partial_t y(t)\|_{L^2(\Omega)}\big)
\le C\big(\|y_0\|_{H^1_0(\Omega;w)} + \|y_1\|_{L^2(\Omega)} + \|f\|_{L^1(0,T;L^2(\Omega))}\big).
\]
\item If $y_0\in D(A)$, $y_1\in H^1_0(\Omega;w)$, and $\partial_t f \in L^1(0,T;L^2(\Omega))$, then
\[
y \in L^\infty(0,T;D(A)) \cap W^{1,\infty}(0,T;H^1_0(\Omega;w)), \quad \partial_{tt}y \in L^\infty(0,T;L^2(\Omega)),
\]
and
\[
\esssup_{0\le t\le T}\Big(\|y(t)\|_{D(A)} + \|\partial_t y(t)\|_{H^1_0(\Omega;w)} + \|\partial_{tt}y(t)\|_{L^2(\Omega)}\Big)
\le C\big(\|y_0\|_{D(A)} + \|y_1\|_{H^1_0(\Omega;w)} + \|f\|_{L^1} + \|\partial_t f\|_{L^1}\big).
\]
\end{enumerate}
\end{theorem}

\subsection{Energy estimates and well-posedness with non-homogeneous Dirichlet boundary input}

We now turn to boundary inputs $y|_{\Gamma}=u(t)$ for the degenerate operator $A$. The key ingredient is a Dirichlet lifting map $D$ that solves $Ay=0$ with prescribed boundary trace, allowing us to reduce the boundary-input problem to a homogeneous one plus a controlled lift. Lemmas below establish existence and boundedness of $D$, boundary/trace regularity, and the estimates needed to prove Theorem~\ref{09.02.T1}.

Denote 
\begin{equation*}
	L^2(\Om;w^{-1})=\left\{y\in \ms{D}'(\Om)\colon \int_\Om y^2w^{-1}\df x<\iy\right\}, 
\end{equation*}
and its inner product is 
\begin{equation*}
	(y,z)_{L^2(\Om;w^{-1})}=\int_\Om yz w^{-1}\df x. 
\end{equation*}

We say that the degenerate elliptic equation 
\begin{equation}\label{07.08.3}
	\begin{cases}
		\mcA u+Vu=f-\Div \bs{f}, &\mbox{in }\Om,\\ 
		u=0, &\mbox{on }\pt\Om 
	\end{cases}
\end{equation}
has a weak solution $u\in H_0^1(\Om;w)$ for $f\in L^2(\Om)$ and $|\bs{f}|\in L^2(\Om;w^{-1})$ with $\bs{f}=(f_1,\cdots, f_N)$ and $0\leq V\in L^\iy(\Om)$ if, for each $\vp\in C_0^\iy(\Om)$, we have
\begin{equation*}
	\int_\Om (\nabla u\cdot \nabla\vp)w\df x+\int_\Om V u\vp\df x=\int_\Om f\vp\df x-\sum_{i=1}^N \int_\Om f_i\f{\pt\vp}{\pt x_i}\df x. 
\end{equation*}

\begin{lemma}\label{07.08.L1'}\label{07.08.L1}
	Under Assumption \ref{06.27.A1}, the degenerate elliptic equation \eqref{07.08.3} has a unique weak solution in $H_0^1(\Om;w)$. Moreover, if $\bs{f}=0$, then $u\in \mcD(\mcA)$ and $u\in H^2(\mcO(\Ga;\be))$. 
\end{lemma}

\begin{proof}
	It is the Lax-Milgram theorem by (2) in Assumption \ref{06.27.A1}.

	If $\bs{f}=0$, then $\mcA u+Vu\in L^2(\Om)$, and $u\in H^2(\mcO(\Ga;\be))$ by (1) in Assumption \ref{06.27.A1} and the standard elliptic regularity (see Theorem 4 in Chapter 6.3 in \cite{Evans} at p.\! 317).
\end{proof}

Now, we will introduce the so-called ``Dirichlet map" in the following. 

\begin{definition}\label{08.20.D1}
	Since for each $s>\f{1}{2}$, we have the Sobolev trace 
	\begin{equation*}
		T: H^s(\Om)\ra H^{s-\f{1}{2}}(\Ga)
	\end{equation*}
	is a surjection. Now, take $s=1$, define 
	\begin{equation*}
		\|u\|_{H^\f{1}{2}(\Ga)}=\inf\left\{\|\psi\|_{H^1(\Om)}\colon \psi\in H^1(\Om), \mbox{and }  T\psi=u\right\}, 
	\end{equation*}
	where $\|\psi\|_{H^1(\Om)}=(\|\nabla u\|_{L^2(\Om)}^2+\|u\|_{L^2(\Om)}^2)^\f{1}{2}$. Then 
	\begin{equation*}
		(H^\f{1}{2}(\Ga), \|\cdot\|_{H^\f{1}{2}(\Ga)})
	\end{equation*}
	is the classical Sobolev space $H^\f{1}{2}(\Ga)$. 
\end{definition}

Now, we introduce the Dirichlet map $D$ by 
\begin{equation}\label{08.14.5}
	D\psi=y, \mbox{ with } 
	\begin{cases}
		A y=0, &\mbox{in }\Om,\\
		y=\psi, &\mbox{on }\Ga, 
	\end{cases}
\end{equation}

\begin{lemma}\label{08.20.L1}
	The operator 
	\begin{equation*}
		D: H^\f{1}{2}(\Ga)\ra H^1(\Om;w).
	\end{equation*}
	that is defined in \eqref{08.14.5} is a bounded linear operator.
	Moreover, $D: H^1(\Ga)\ra H^1(\Om;w)$ is a bounded linear operator. 
\end{lemma}

\begin{proof} 
	It is a well-known result of elliptic theory and  \eqref{08.13.1}, there exists $\wt \psi\in H^1(\Om)$ such that $\|\wt \psi\|_{H^1(\Om)}\leq 2\|\psi\|_{H^\f{1}{2}(\Ga)}$. Choosing $\zeta\in C^\iy(\R^N)$ with 
	\begin{equation*}
		\zeta=1 \mbox{ on } \ol{\mcO(\Ga;\be)}, \quad \zeta=0 \mbox{ on } \Om\se \mcO(\Ga; 2\be), \mbox{ and } |\nabla \zeta|\leq C\be^{-1}. 
	\end{equation*} 
	Let $\wh \psi=\zeta \wt \psi$. Then  $\supp \wh \psi\s \ol{\mcO(\Ga; 2\be)}$ such that $\wh \psi=\psi$ on $\Ga$ in the sense of trace and $\|\wh \psi\|_{H^1(\Om)}\leq C\|\wt \psi\|_{H^1(\Om)}$. Note that  the equation 
	\begin{equation*}
		\begin{cases}
			\mcA \wh y=-A \wh \psi, &\mbox{in }\Om,\\
			\wh y=0, &\mbox{on }\Ga
		\end{cases}
	\end{equation*}
	has a unique solution $\wh y\in H_0^1(\Om;w)$ by Lemma \ref{07.08.L1}, and 
	\begin{equation*}
		\int_\Om |\nabla \wh y|^2w\df x\leq C\int_\Om |\nabla \wh \psi|^2\df x, 
	\end{equation*}
	where the constant $C>0$ depends only on $\sup_{x\in \Om}w, \be$ and $\Om$. 
	Taking $y=\wh y+\wh \psi$, then $D\psi=y$, and 
	\begin{equation}\label{08.15.1}
		\begin{split} 
			\int_\Om |\nabla y|^2w\df x+\int_\Om y^2\df x 
			&\leq C\int_\Om |\nabla \wh y|^2w\df x+C\int_\Om |\nabla \wh \psi|^2\df x+2\int_\Om |\wh \psi|^2\df x\\
			&\leq C\|\wh \psi\|_{H^1(\Om)}^2\leq C\|\psi\|_{H^\f{1}{2}(\Ga)}^2
		\end{split} 
	\end{equation}
	according to (2) in Assumption \ref{06.27.A1}, where the constants $C>0$ depends only on the constant in the Poincar\'e inequality, $\sup_{x\in \Om}w, \be$ and $\Om$. 
\end{proof}

\begin{lemma}\label{09.11.L1}
	Under Assumption \ref{06.27.A1}, the degenerate elliptic equation \begin{equation}\label{09.11.1}
		\begin{cases}
			Av=f, &\mbox{in }\Om,\\
			v=\psi, &\mbox{on }\pt\Ga, 
		\end{cases}
	\end{equation}
	has a weak solution in $H^1(\Om;w)$, where $f\in H^{-1}(\Om;w)$ and $\psi\in H^1(\Ga)$. Moreover,  
	\begin{equation}\label{09.12.4}
		\|v\|_{H^1(\Om;w)}\leq C\left(\|\psi\|_{H^1(\Ga)}+\|f\|_{H^{-1}(\Om;w)}\right). 
	\end{equation}
	Furthermore, if $f\in L^2(\Om)$, then  $v\in \mcD(A)$ and $v\in H^2(\mcO(\Ga;\be))$. 
\end{lemma}

\begin{proof}
	From Lemma \ref{07.08.L1'}, let $z\in H_0^1(\Om;w)$ satisfies the equation
	\begin{equation*}
		\begin{cases}
			\mcA z=f, &\mbox{in }\Om,\\
			z=0, &\mbox{on }\Ga, 
		\end{cases}
	\end{equation*}
	take $v=z+D\psi$, then   $v$ is the solution of \eqref{09.11.1}, where  $D\psi$ is defined in \eqref{08.14.5}.    Moreover, we have the estimate \eqref{09.12.4} from  Lemma \ref{07.08.L1'} and \eqref{08.15.1}. 
\end{proof}

In this subsection, we estimate the solution of \eqref{09.02.1} with non-homogeneous boundary condition. This follows the ideas of \cite{Bellassoued,Lasiecka1,Lasiecka2,Lions} with minor modifications.

We consider the following equation
\begin{equation}\label{09.02.1}
	\begin{cases}
		\pt_{tt}y+A y=f, &\mbox{in }Q,\\
		y(0)=y_0, \pt_ty(0)=y_1, &\mbox{in }\Om,\\
		y=u, &\mbox{on }\Sigma, 
	\end{cases}
\end{equation}
where $f, y_0, y_1, u$ will be given later.

\begin{theorem}\label{09.02.T1}
	Under Assumption \ref{06.27.A1}. Suppose that 
	\begin{equation*}
		f\in L^1(0,T; L^2(\Om)),\quad y_0\in H_0^1(\Om;w),\quad y_1\in L^2(\Om), \mbox{ and } u\in H^1(\Sigma). 
	\end{equation*} 
	Let $y$ be the solution of \eqref{09.02.1}. 	
	Then the solution
	\begin{equation*}
		y\in C([0,T]; H^1(\Om;w))\cap C^1([0,T]; L^2(\Om))
	\end{equation*}
	of \eqref{09.02.1} satisfies: there exists a constant $C>0$, depends only on $w, T$ and $\Om$,  such that 
	\begin{equation}\label{09.04.1}
		\begin{split} 
		&\|y(t)\|_{C([0,T];H^1(\Om;w))}+\|\pt_t y(t)\|_{C([0,T]; L^2(\Om))}\\
		&\leq C\left(\|u\|_{H^1(\Sigma)}+\|y_0\|_{H_0^1(\Om;w)}+\|y_1\|_{L^2(\Om)}+\|f\|_{L^1(0,T; L^2(\Om))}\right).
		\end{split} 
	\end{equation}
	Furthermore, we have 
	\begin{equation*}
		\f{\pt y}{\pt \nu_\mcA}\in L^2(\Sigma), 
	\end{equation*}
	and there exists a constant $C=C(T,\Om)>0$, such that 
	\begin{equation}\label{09.04.2}
		\|\pt_{\nu_\mcA} y\|_{L^2(\Sigma)}\leq C\left(\|u\|_{H^1(\Sigma)}+\|y_0\|_{H^1(\Om;w)}+\|y_1\|_{L^2(\Om)}+\|f\|_{L^1(0,T; L^2(\Om))}\right).
	\end{equation}
\end{theorem}

\begin{lemma}\label{09.04.L5}
	There exists a smooth vector field $\bs{n}$ such that 
	\begin{equation*}
		\bs{n}(x)=\nu(x) \mbox{ for } x\in \pt\Om,  \mbox{ and } |\bs{n}(x)|\leq 1 \mbox{ for } x\in \Om, 
	\end{equation*}
	where $\nu$ is the unit outward normal vector to $\pt\Om$. 
\end{lemma}

\begin{proof}
	This is Lemma 2.3 (p.\! 61) in Chapter 2.6 in \cite{Bellassoued}. 
\end{proof}

\begin{lemma}\label{09.04.L3}
	Under the conditions in Theorem \ref{09.02.T1} with $u=0$ on $\Sigma$. Then the mapping 
	\begin{equation*}
		H_0^1(\Om;w)\ts L^2(\Om)\ts L^1(0,T; L^2(\Om))\ra L^2(\Sigma), \ (y_0,y_1,f)\mapsto \f{\pt y}{\pt\nu_\mcA}
	\end{equation*}
	is a continuous linear operator, where $y$ is the solution of \eqref{09.02.1} with $u=0$ on $\Sigma$. Furthermore, there exists a constant $C>0$ such that 
	\begin{equation}\label{09.04.6}
		\|\pt_{\nu_\mcA} y\|_{L^2(\Sigma)}\leq C \left(\|y_0\|_{H_0^1(\Om;w)}+\|y_1\|_{L^2(\Om)}+\|f\|_{L^1(0,T; L^2(\Om))}\right).
	\end{equation}
\end{lemma}

\begin{proof}
	By the density argument, it suffices to prove \eqref{09.04.6} for $y_0\in \mcD(A), y_1\in H_0^1(\Om;w)$ and $f\in L^1(0,T;L^2(\Om))$ with $\pt_t f\in L^1(0,T; L^2(\Om))$. By Lemma \ref{09.04.L5}, there exists a smooth vector field $\bs{n}$ on $\Om$ such that 
	\begin{equation*}
		\bs{n}(x)=\nu(x) \mbox{ for } x\in\pt\Om,  \mbox{ and } \supp \bs{n}\s \ol{\mcO({\Ga;\be})},
	\end{equation*}
	and 
	\begin{equation*}
		|\bs{n}(x)|\leq 1 \mbox{ for } x\in \Om,\mbox{ and } |\Div \bs{n}(x)|\leq C\be^{-1} \mbox{ for } x\in \Om, 
	\end{equation*}
	where the constant $C>0$ is an absolute constant, for otherwise, we multiply $\bs{n}$ by a cut-off function $\zeta\in C_0^\iy(\R^N)$ with $\zeta=1$ on $\mcO(\Ga;\f{1}{2}\be)$ and $\zeta=0$ on $\Om\se \mcO(\Ga;\be)$, and then $\zeta\bs{n}$ is the desired function. 
	Multiplying $\bs{n}\cdot \nabla y$ on the both sides of \eqref{09.02.1}, integrating on $Q$, we have 
	\begin{equation*}
		\begin{split}
			I:=\iint_Q f(\bs{n}\cdot\nabla y)\df x\df t=\iint_Q (\pt_{tt}y)(\bs{n}\cdot \nabla y)\df x\df t-\iint_Q \Div(w\nabla y) (\bs{n}\cdot \nabla y)\df x\df t=:I_1+I_2. 
		\end{split}
	\end{equation*}
	
	Integrating by parts, remembering $\supp \bs{n}\s \ol{\mcO(\Ga;\be)}$, we have
	\begin{equation*}
		\begin{split}
			I_1
			&=\int_\Om (\pt_ty)(\bs n\cdot \nabla y)\df x\bigg|_{t=0}^{t=T}-\f{1}{2}\iint_Q \bs{n}\cdot \nabla |\pt_ty|^2\df x\df t\\
			&=\int_\Om (\pt_ty)(\bs n\cdot \nabla y)\df x\bigg|_{t=0}^{t=T}-\f{1}{2}\iint_\Sigma |\pt_ty|^2\df S\df t+\f{1}{2}\iint_Q |\pt_ty|^2\Div\bs{n}\df x\df t, 
		\end{split}
	\end{equation*}
	and then, by $\pt_ty=0$ on $\Sigma$ and \eqref{09.04.1}, we have 
	\begin{equation}\label{09.05.1}
		\begin{split}
			|I_1|
			&\leq C\left(\|y_0\|_{H_0^1(\Om;w)}^2+\|y_1\|_{L^2(\Om)}^2+\|f\|_{L^1(0,T; L^2(\Om))}^2\right). 
		\end{split}
	\end{equation}
	
	Integrating by parts, and $y=0$ on $\Sigma$, we have 
	\begin{equation*}
		\begin{split}
			I_2
			&=-\iint_\Sigma \f{1}{w}|\pt_{\nu_\mcA} y|^2\df S\df t+\iint_Q w\nabla y\cdot \nabla (\bs{n}\cdot \nabla y)\df x\df t\\
			&=-\f{1}{2}\iint_\Sigma \f{1}{w}|\pt_{\nu_\mcA} y|^2\df S\df t+\iint_Q w \left(\nabla y\cdot [(D\bs{n})\nabla y]\right) \df x\df t\\
			&\hspace{4.5mm}-\f{1}{2}\iint_Q (\bs{n}\cdot \nabla w)|\nabla y|^2\df x\df t-\f{1}{2}\iint_Q w|\nabla y|^2\Div \bs{n}\df x\df t. 
		\end{split}
	\end{equation*}
	Combining \eqref{09.05.1} and $I=I_1+I_2$, and using $w\geq \Lambda$ on $\Gamma$, we obtain 
	\begin{equation*}
		\begin{split}
			\iint_\Sigma |\pt_{\nu_\mcA} y|^2\df S\df t
			&\leq C\left(\|y_0\|_{H_0^1(\Om;w)}+\|y_1\|_{L^2(\Om)}+\|f\|_{L^1(0,T;L^2(\Om))}\right)^2,
		\end{split}
	\end{equation*}
	where the constant $C>0$ depends only on $\bs{n}, w, \La, \be, T$ and $\Om$ and the constant in (2) in Assumption \ref{06.27.A1}. 
\end{proof}

\begin{definition}\label{09.04.D1}
	Under Assumption \ref{06.27.A1}, assume that 
	\begin{equation*}
		f=0 \mbox{ in }Q, \quad y_0\in L^2(\Om), \quad y_1\in H^{-1}(\Om;w), \mbox{ and } u\in L^2(\Sigma). 
	\end{equation*}
	We call 
	\begin{equation*}
		y\in C([0,T]; L^2(\Om)), \mbox{ with } \pt_ty \in C([0,T]; H^{-1}(\Om;w)), 
	\end{equation*} 
	is a solution of \eqref{09.02.1}  if,  for each $t\in [0,T]$, and for all  $z_0\in H_0^1(\Om),z_1\in L^2(\Om)$, we have 
	\begin{equation}\label{09.12.1}
		M(z_0,z_1)=(y(t), \pt_t z(t))_{L^2(\Om)}-\lg \pt_ty(t), z(t)\rg_{H^{-1}(\Om;w), H_0^1(\Om;w)}, 
	\end{equation}
	where 
	\begin{equation*}
		M(z_0,z_1)=\int_0^t\int_\Ga  \f{\pt z}{\pt \nu_\mcA}(x,s) u(x,s)\df S\df s+(y_0, z_1)_{L^2(\Om)}-\lg y_1, z_0\rg_{H^{-1}(\Om;w), H_0^1(\Om;w)}, 
	\end{equation*}
	and $z$ is a solution of 
	\begin{equation*}
		\begin{cases}
			\pt_{tt}z+\mcA z=0, &\mbox{in }Q,\\
			z(0)=z_0, \pt_tz(0)=z_1, &\mbox{in }\Om,\\
			z=0, &\mbox{on }\Sigma. 
		\end{cases}
	\end{equation*} 
\end{definition}

\begin{lemma}\label{09.04.L1}
	Under the conditions in Definition \ref{09.04.D1}. Then there exists a unique solution 
	\begin{equation*}
		y\in C([0,T]; L^2(\Om))\cap C^1([0,T]; H^{-1}(\Om;w))
	\end{equation*}
	of \eqref{09.02.1}, and there exists a constant $C>0$, depends only on $w$, such that for $t\in (0,T)$ we have 
	\begin{equation*}
		\|y(t)\|_{L^2(\Om)}+\|\pt_t y(t)\|_{H^{-1}(\Om;w)}\leq C\left(\|y_0\|_{L^2(\Om)}+\|y_1\|_{H^{-1}(\Om;w)}+\|u\|_{L^2(\Sigma)}\right). 
	\end{equation*} 
\end{lemma}

\begin{proof}
	It is easily verified that 
	\begin{equation}\label{09.12.2}
		|M(z_0,z_1)|\leq C\left(\|y_0\|_{L^2(\Om)}+\|y_1\|_{H^{-1}(\Om;w)}+\|u\|_{L^2(\Sigma)}\right)\left(\|z_0\|_{H_0^1(\Om;w)}+\|z_1\|_{L^2(\Om)}\right) 
	\end{equation}
	by Lemma \ref{09.04.L3}, note that 
	\begin{equation}\label{09.12.3}
		H_0^1(\Om;w)\ts L^2(\Om)\ra H_0^1(\Om;w)\ts L^2(\Om), \ (z(t), \pt_tz(t))\mapsto (z_0,z_1) 
	\end{equation}
	is an isomorphism, Hence the linear form 
	\begin{equation*}
		(z(t), \pt_tz(t))\mapsto M(z_0,z_1)
	\end{equation*}
	is also bounded on $H_0^1(\Om;w)\ts L^2(\Om)$. Hence, there exists a unique pair $(y(t), \pt_ty(t))\in L^2(\Om)\ts H^{-1}(\Om;w)$ satisfying \eqref{09.12.1}. 
	
	Finally, from \eqref{09.12.1} and \eqref{09.12.2}, we have 
	\begin{equation*}
		\begin{split}
			&\left|(y(t), \pt_t z(t))_{L^2(\Om)}-\lg \pt_ty(t), z(t)\rg_{H^{-1}(\Om;w), H_0^1(\Om;w)}\right|\\
			&\leq C\left(\|y_0\|_{L^2(\Om)}+\|y_1\|_{H^{-1}(\Om;w)}+\|u\|_{L^2(\Sigma)}\right)\left(\|z_0\|_{H_0^1(\Om;w)}+\|z_1\|_{L^2(\Om)}\right), 
		\end{split}
	\end{equation*}
	Note again that \eqref{09.12.3} is an isomorphism; then we get
	\begin{equation*}
		\|y(t)\|_{L^2(\Om)}+\|\pt_t y(t)\|_{H^{-1}(\Om;w)}\leq C\left(\|y_0\|_{L^2(\Om)}+\|y_1\|_{H^{-1}(\Om;w)}+\|u\|_{L^2(\Sigma)}\right). 
	\end{equation*} 
	This completes the proof of the lemma. 
\end{proof}

Now, we are in position to prove Theorem \ref{09.02.T1}. 

\begin{proof}[Proof of Theorem \ref{09.02.T1}]
	It is easy to see that the solution of the equation \eqref{09.02.1} can be written as 
	\begin{equation*}
		y=h+z, 
	\end{equation*}
	where $h$ is the solution of the following system
	\begin{equation}\label{09.12.5}
		\begin{cases}
			\pt_{tt}h+\mcA h=f, &\mbox{in }Q, \\
			h(0)=\pt_th(0)=0, &\mbox{in }\Om,\\
			h=0, &\mbox{on }\Sigma, 
		\end{cases}
	\end{equation}
	and $z$ is the solution of the following system
	\begin{equation}\label{09.12.6}
		\begin{cases}
			\pt_{tt}z+A z=0, &\mbox{in }Q, \\
			z(0)=y_0, \pt_tz(0)=y_1, &\mbox{in }\Om,\\
			z=u, &\mbox{on }\Sigma. 
		\end{cases}
	\end{equation}
	
	On one hand, for the system \eqref{09.12.5}, we have Theorem \ref{thm:2.12} and Lemma \ref{09.04.L3}. i.e., 
	\begin{equation*}
		\|h(t)\|_{C([0,T];H^1(\Om;w))}+\|\pt_t h(t)\|_{C([0,T]; L^2(\Om))}\leq C\|f\|_{L^1(0,T; L^2(\Om))}, 
	\end{equation*}
	and 
	\begin{equation*} 
		\|\pt_{\nu_\mcA} h\|_{L^2(\Sigma)}\leq C \|f\|_{L^1(0,T; L^2(\Om))}.
	\end{equation*}
	
	On the other hand, for system \eqref{09.12.6}, we set $\wh z=\pt_tz$, and 
	\begin{equation*}
		\begin{cases}
			\pt_{tt}\wh z+A \wh z=0, &\mbox{in }Q, \\
			\wh z(0)=y_1, \pt_t\wh z(0)=\mcA y_0, &\mbox{in } \Om,\\
			\wh z=\pt_t u, &\mbox{on }\Sigma, 
		\end{cases}
	\end{equation*}
	since $\pt_tu\in L^2(\Sigma), y_1\in L^2(\Om)$, and $\mcA y_0\in H^{-1}(\Om;w)$ by Corollary \ref{09.03.C1} in the following, by Lemma \ref{09.04.L1}, we have 
	\begin{equation*}
		\wh z\in C([0,T]; L^2(\Om))\cap C^1([0,T]; H^{-1}(\Om;w)), 
	\end{equation*}
	and there exists a constant $C>0$ such that for all $t\in (0,T)$, 
	\begin{equation*}
		\|\wh z\|_{L^2(\Om)}+\|\pt_t\wh z(t)\|_{H^{-1}(\Om;w)}\leq C\left(\|y_0\|_{H_0^1(\Om;w)}+\|y_1\|_{L^2(\Om)}+\|u\|_{H^1(\Sigma)}\right)
	\end{equation*}
	by Corollary \ref{09.03.C1} in the following. 
	These imply that 
	\begin{equation*}
		z\in C^1([0,T]; L^2(\Om))\cap C^2([0,T]; H^{-1}(\Om;w)), 
	\end{equation*}
	and from $\pt_{tt}z+A z=0$ we have 
	\begin{equation*}
		Az\in C([0,T]; H^{-1}(\Om;w)). 
	\end{equation*}
	Note that $u(\cdot, t)\in H^1(\Ga)$, from Lemma \ref{09.11.L1}, we have 
	\begin{equation*}
		z\in C([0,T]; H^1(\Om;w))\cap C^1([0,T]; L^2(\Om)). 
	\end{equation*}
	Moreover, there exists $C>0$ such that for all $t\in (0,T)$ we have 
	\begin{equation*}
		\|\pt_tz\|_{L^2(\Om)}+\|Az\|_{H^{-1}(\Om;w)}\leq C\left(\|y_0\|_{H_0^1(\Om;w)}+\|y_1\|_{L^2(\Om)}+\|u\|_{H^1(\Sigma)}\right). 
	\end{equation*}
	Using Lemma \ref{09.11.L1} again, we obtain 
	\begin{equation*}
		\|z(t)\|_{H^1(\Om;w)}+\|\pt_t z(t)\|_{L^2(\Om)}\leq C\left(\|y_0\|_{H_0^1(\Om;w)}+\|y_1\|_{L^2(\Om)}+\|u\|_{H^1(\Sigma)}\right). 
	\end{equation*}
	
	Overall, we have proved Theorem \ref{09.02.T1}. 
\end{proof}

\begin{remark}\label{rem:lifted-setting}
The a priori estimates \eqref{09.04.1}--\eqref{09.04.2} yield continuous dependence and uniform boundedness of solutions with respect to the data $(y_0,y_1,f,u)$, and we will cite them as the stability interface for the boundary-input problem. In the sequel, when treating non-homogeneous Dirichlet inputs, we work with the lifted representation $y=z+Du$, where the Dirichlet map $D$ is defined in Subsection 2.2 (see \eqref{08.14.5}). This reduction places the boundary-input problem into the homogeneous-boundary plus forcing framework, and it aligns directly with the mild-solution representations developed in Subsections 2.3 and 2.4.
\end{remark}

\subsection{Mild solution with Dirichlet boundary condition}

In this subsection we cast the homogeneous Dirichlet problem into an abstract evolution equation generated by the degenerate elliptic operator. This yields the cosine operator framework and the standard mild solution representation used later.

In this section, we shall establish the mild solution for equation
\begin{equation}\label{08.14.3}
	\begin{cases} 
		\pt_{tt}y-\Div(w\nabla y)=0, &\mbox{in }Q,\\
		y(0)=y_0, \pt_t y(0)=y_1, &\mbox{in }\Om,\\
		y=0, &\mbox{on }\Sigma, 
	\end{cases}
\end{equation}
where $y_0\in H_0^1(\Om;w)$ and $y_1\in L^2(\Om)$. 

Before establishing the mild solution for the equation \eqref{08.14.3}, some preparatory steps are necessary.

\begin{definition}\label{08.13.D1}
	For each $0\leq \theta \in\R$, we denote 
	\begin{equation*}
		\mcD(\mcA^\theta)=\left\{u\in L^2(\Om)\colon \sum_{n\in\N}\la_n^{2 \theta}(u,\Phi_n)_{L^2(\Om)}^2<+\iy \right\}.
	\end{equation*}  
\end{definition}

\begin{definition}\label{08.13.D2}
	For $\theta\in [0,\f{1}{2}]$, we define 
	\begin{equation*}
		\mcD(\mcA^{-\theta})=\left\{u\in\ms{D}'(\Om)\colon \sum_{n\in\N}\la_n^{-2\theta}\lg u,\Phi_n\rg_{\mcD'(\mcA^\theta), \mcD(A^\theta)}^2<+\iy\right\}, 
	\end{equation*}
	where  $\ms{D}(\Om)=C_0^\iy(\Om)$ and $\ms{D}'(\Om)$ is the space of distributions, and $\mcD'(\mcA^\theta)$ is the dual space of $\mcD(\mcA^\theta)$. 
\end{definition}

\begin{remark}\label{08.13.R1}
	From Definition \ref{08.13.D1} and Lemma \ref{06.28.L1}, it is obviously that 
	\begin{equation*}
		\mcD(\mcA^0)=L^2(\Om),\quad \mcD(\mcA^1)=\mcD(\mcA), \mbox{ and } \mcD(\mcA^\f{1}{2})=H_0^1(\Om;w), 
	\end{equation*}
	and 
	\begin{equation*}
		\Phi_n\in \ba_{n=0}^\iy \mcD(\mcA^n) \mbox{ for all } n\in\N. 
	\end{equation*}
	Moreover, (1) For all $0\leq \theta\leq \eta$, we have
	\begin{equation*}
		\mcD(\mcA^\eta)\s \mcD(\mcA^\theta).
	\end{equation*}
	This is a conclusion of Lemma \ref{07.08.L1} and 
	\begin{equation*}
		\sum_{n\in\N}\la_n^{2\eta}(u,\Phi_n)_{L^2(\Om)}^2\geq \la_1^{2(\eta-\theta)}\sum_{n\in\N}\la_n^{2\theta}(u,\Phi_n)_{L^2(\Om)}^2. 
	\end{equation*}
	
	(2) $\mcD(\mcA)$ is dense in $L^2(\Om)$ and $H_0^1(\Om;w)$. Indeed, for each $u=\sum_{n=1}^\iy u_n\Phi_n\in L^2(\Om)$ with $u_n=(u, \Phi_n)_{L^2(\Om)}$, take $u_m=\sum_{n=1}^m u_n\Phi_n$, then $u_m\in \mcD(\mcA)$ and $\|u_m-u\|_{L^2(\Om)}\ra 0$ as $m\ra\iy$. Next, for each $u=\sum_{n=1}^\iy u_n\Phi_n\in H_0^1(\Om;w)$ with $u_n=(u,\Phi_n)_{L^2(\Om)}$,  take $u_m=\sum_{n=1}^m u_n\Phi_n$, then $u_m\in\mcD(\mcA)$ and from Lemma \ref{06.28.L1} we have 
	\begin{equation*}
		\|u_m-u\|_{H_0^1(\Om;w)}^2=\int_\Om \left|\sum_{n=m+1}^\iy \la_n^\f{1}{2}u_n \left(\la_n^{-\f{1}{2}}\nabla \Phi_n\right)\right|^2w\df x=\sum_{n=m+1}^\iy u_n^2\la_n \ra 0 \mbox{ as } m\ra\iy. 
	\end{equation*}
	
	(3) Eigenfunction $\Phi_n\ (n\in\N)$ may not belong to $H^1(\Om)$. And $C_0^\iy(\Om)$ is not necessarily included in  $\mcD(\mcA)$. 
\end{remark}

\begin{remark}\label{08.13.R2}
	(1) For $\theta\in [0,\f{1}{2}]$, we have $\mcD(\mcA^{-\theta})=\mcD'(\mcA^\theta)$,  especially,  $H^{-1}(\Om;w)=\mcD(\mcA^{-\f{1}{2}})$. 
	
	Indeed, denote  $u=\sum_{n=1}^\iy u_n\Phi_n\in L^2(\Om)$ with $u_n=(u,\Phi_n)_{L^2(\Om)}$ for all $n\in\N$.  From Definition \ref{08.13.D1}, we have $u\in \mcD(\mcA^\theta)$ if and only if $\|u\|_{\mcD(\mcA^\theta)}^2=\sum_{n=1}^\iy u_n^2\la_n^{2\theta}<\iy$. 
	
	Now, for each $\vp\in \mcD'(\mcA^{\theta})$, taking $m\in\N$, then 
	\begin{equation*}
		u_n^m=\f{\la_n^{-2\theta}\lg \vp,\Phi_n\rg_{\mcD'(\mcD^\theta), \mcD(\mcA^\theta)}}{\left(\sum_{n=1}^m \la_n^{-2\theta}\lg\vp,\Phi_n\rg_{\mcD'(\mcA^\theta), \mcD(\mcA^\theta)}^2\right)^\f{1}{2}} \mbox{ for all } n\leq m, \quad u^m=\sum_{n=1}^m u_n^m\Phi_n, 
	\end{equation*}
	then $u^m\in L^2(\Om)$  for all $m\in\N$, and from Definition \ref{08.13.D1} we have 
	\begin{equation*}
		\sum_{n=1}^m (u_n^m)^2\la_n^{2\theta}=1\Lra u^m\in \mcD(\mcA^\theta). 
	\end{equation*}
	And hence for all $m\in\N$, 
	\begin{equation*}
		\begin{split}
			\|\vp\|_{\mcD'(\mcA^{\theta})}
			&=\sup_{
				\sum_{n=1}^\iy u_n^2\la_n^{2\theta}\leq 1}\lg \vp, u\rg_{\mcD'(\mcA^\theta), \mcD(\mcA^\theta)}\\
			&\geq \lg\vp, u^m\rg_{\mcD'(\mcA^\theta), \mcD(\mcA^\theta)}=\left(\sum_{n=1}^m \la_n^{-2\theta}\llg \vp, \Phi_n\rrg_{\mcD'(\mcA^\theta), \mcD(\mcA^\theta)}^2\right)^\f{1}{2}.
		\end{split}
	\end{equation*}
	This implies that 
	\begin{equation*}
		\|\vp\|_{\mcD'(A^\theta)}\geq \left(\sum_{n=1}^\iy \la_n^{-2\theta}\llg \vp, \Phi_n\rrg_{\mcD'(\mcA^\theta), \mcD(\mcA^\theta)}^2\right)^\f{1}{2}
	\end{equation*}
	by the arbitrary of $m\in\N$. 
	Hence $\mcD'(\mcA^\theta)\s \mcD(\mcA^{-\theta})$. 
	
	On the other hand, for each $\vp\in \mcD(\mcA^{-\theta})$,  
	we define
	\begin{equation}\label{08.14.1}
		\begin{split}
			\vp: \mcD(\mcA^\theta)\ra \R, \ \lg \vp, u\rg 
			&=\sum_{n=1}^\iy u_n\lg\vp, \Phi_n \rg_{\mcD'(\mcA^{\theta}),\mcD(\mcA^\theta)}, 
		\end{split}
	\end{equation}
	where $u=\sum_{n=1}^\iy u_n\Phi_n$ with $u_n=(u,\Phi_n)_{L^2(\Om)}$ for all $n\in\N$, 
	it is easily verified that  
	\begin{equation*}
		\begin{split}
			\left|\lg \vp, u\rg\right| 
			&= \left|\sum_{n=1}^\iy u_n\la_n^\theta\la_n^{-\theta}\llg \vp, \Phi_n\rrg_{\mcD'(\mcA^{\theta}),\mcD(\mcA^\theta)}\right|\leq \|u\|_{\mcD(\mcA^\theta)}\left(\sum_{n=1}^\iy \la_n^{-2\theta}\lg \vp, \Phi_n\rg_{\mcD'(\mcA^{\theta}),\mcD(\mcA^\theta)}^2\right)^\f{1}{2}
		\end{split}
	\end{equation*}
	from Cauchy-Schwarz inequality, then 
	$\vp$ is a bounded linear functional on $\mcD(\mcA^\theta)$, and hence $\mcD(\mcA^{-\theta})\s\mcD'(\mcA^\theta)$. Moveover,  
	\begin{equation*}
		\begin{split}
			\|\vp\|_{\mcD'(\mcA^{\theta})}
			&=\sup_{
				\sum_{n=1}^\iy u_n^2\la_n^{2\theta}\leq 1}\sum_{n=1}^\iy u_n\la_n^\theta\la_n^{-\theta}\llg \vp, \Phi_n\rrg_{\mcD'(\mcA^{\theta}),\mcD(\mcA^\theta)}\leq \left(\sum_{n=1}^\iy \la_n^{-2\theta}\llg \vp, \Phi_n\rrg_{\mcD'(\mcA^{\theta}),\mcD(\mcA^\theta)}^2\right)^\f{1}{2}.
		\end{split}
	\end{equation*}

	From above, we get 
	\begin{equation*}
		\|\vp\|_{\mcD(\mcA^{-\theta})}=\left(\sum_{n=1}^\iy \la_n^{-2\theta}\llg \vp, \Phi_n\rrg_{\mcD'(\mcA^{\theta}),\mcD(\mcA^\theta)}^2\right)^\f{1}{2}.
	\end{equation*} 
	
	(2) For $\theta>\f{1}{2}$, it is hard to define $\mcD(\mcA^{-\theta})$ since $C_0^\iy(\Om)\not\s \mcD(\mcA^\theta)$ in general.  This highlights a key distinction between classical Sobolev spaces and weighted Sobolev spaces.
\end{remark}

The following lemma is an extended version of Lemma \ref{06.28.L1}.

\begin{lemma}\label{09.03.L2}\label{06.28.L1}
	Let $\f{1}{2}\leq \theta\in\R$ be a fixed constant. Then the operator 
	\begin{equation*}
		\mcA: \mcD(\mcA^\theta)\ra \mcD(\mcA^{\theta-1}), 
	\end{equation*}
	is isometric linear isomorphism. Moreover, for each $u=\sum_{n=1}^\iy (u,\Phi_n)_{L^2(\Om)}\Phi_n\in \mcD(\mcA^\theta)$, we have 
	\begin{equation}\label{09.03.7}
		\mcA\left(\sum_{n=1}^\iy (u,\Phi_n)_{L^2(\Om)}\Phi_n\right)=\sum_{n=1}^\iy (u, \Phi_n)_{L^2(\Om)}\la_n\Phi_n. 
	\end{equation}
\end{lemma}

\begin{proof}
	Take $\theta=\f{1}{2}$. For each $u,v\in \mcD(\mcA^\f{1}{2})= H_0^1(\Om;w)$, note that $\mcD(\mcA^{-\f{1}{2}})=\mcD'(\mcA^\f{1}{2})=H^{-1}(\Om;w)$,  we have 
	\begin{equation*}
		\begin{split} 
			\lg \mcA u, v\rg_{\mcD(\mcA^{-\f{1}{2}}),\mcD(\mcA^\f{1}{2})}
			&=\int_\Om (\nabla u\cdot \nabla v)w\df x
			=\sum_{n=1}^\iy (u_n,\Phi_n)_{L^2(\Om)}(v_n,\Phi_n)_{L^2(\Om)}\la_n\\
			&\leq \|u\|_{H_0^1(\Om;w)}\|v\|_{H_0^1(\Om;w)}, 
		\end{split}
	\end{equation*}
	and hence 
	\begin{equation}\label{09.03.5}
		\|\mcA u\|_{\mcD(\mcA^{-\f{1}{2}})}\leq \|u\|_{H_0^1(\Om;w)}. 
	\end{equation}
	Moreover, we have 
	\begin{equation}\label{09.03.6}
		\mcA u_m\ra \mcA u \mbox{ strongly in } \mcD(\mcA^{-\f{1}{2}}) \mbox{ with } u_m=\sum_{n=1}^m (u,\Phi_n)_{L^2(\Om)}\Phi_n \mbox{ for } m\in\N. 
	\end{equation}
	Hence, for $0\neq u\in \mcD(\mcA^\f{1}{2})$, for $m\in\N$ large enough, we have 
	\begin{equation*}
		\begin{split}
			\|\mcA u_m\|_{\mcD(\mcA^{-\f{1}{2}})}
			&=\sup_{\|v\|_{H_0^1(\Om;w)}\leq 1}\lg \mcA u_m, v\rg_{\mcD(\mcA^{-\f{1}{2}}),\mcD(\mcA^\f{1}{2})}\geq \llg \mcA u_m, \f{u_m}{\|u\|_{H_0^1(\Om;w)}}\rrg_{\mcD(\mcA^{-\f{1}{2}}),\mcD(\mcA^\f{1}{2})}\\
			&=\f{1}{\|u\|_{H_0^1(\Om;w)}}\sum_{n=1}^m (u_n,\Phi_n)^2\la_n\geq \left(\sum_{n=1}^m (u_n,\Phi_n)^2\la_n\right)^\f{1}{2}=\|u_m\|_{H_0^1(\Om;w)}.
		\end{split}
	\end{equation*}
	This shows that $\|\mcA u\|_{\mcD(\mcA^{-\f{1}{2}})}\geq \|u\|_{H^1(\Om;w)}$. This together with \eqref{09.03.5} we get $\|u\|_{\mcD(\mcA^{-\f{1}{2}})}=\|u\|_{\mcD(\mcA^\f{1}{2})}$ for all $u\in \mcD(\mcA^\f{1}{2})$. Again, from 
	\begin{equation*}
		\mcA u_m=\sum_{n=1}^m (u,\Phi_n)_{L^2(\Om)}\la_n\Phi_n
	\end{equation*}
	and \eqref{09.03.6} we obtain \eqref{09.03.7}.
	
	Let $\theta\geq \f{1}{2}$. For $u=\sum_{n=1}^\iy (u, \Phi_n)_{L^2(\Om)}\Phi_n\in \mcD(\mcA^\theta)$, take $u_m=\sum_{n=1}^m (u,\Phi_n)_{L^2(\Om)}\Phi_n$ for all $m\in\N$,   we have $\mcA u_m=\sum_{n=1}^m (u,\Phi_n)_{L^2(\Om)}\la_n \Phi_n$, and for $k< m$ we have 
	\begin{equation*}
		\begin{split} 
			\|\mcA u_m-\mcA u_k\|_{\mcD(\mcA^{\theta-1})}^2=\sum_{n=k+1}^m(u,\Phi_n)_{L^2(\Om)}^2\la_n^2 \la_n^{2\theta-2}= \|u_m-u_k\|_{\mcD(\mcA^\theta)}^2\ra 0 \mbox{ as } k,m\ra \iy. 
		\end{split} 
	\end{equation*}
	This implies that $\mcA u_m\ra h$ strongly in $\mcD(\mcA^{\theta-1})$. Now,  by $\mcD(\mcA^\theta)\s \mcD(\mcA^\f{1}{2})$ since $\theta\geq \f{1}{2}$, we have $u_m\ra u$ strongly in $\mcD(\mcA^\f{1}{2})$ from (1) in Remark \ref{08.13.R1}, and then  for each $\vp\in C_0^\iy(\Om)$, we have 
	\begin{equation*}
		\begin{split}
			\lg h, \vp\rg
			&=\lim_{m\ra\iy}\lg \mcA u_m, \vp\rg=\lim_{m\ra\iy}\int_\Om (\nabla u_m\cdot \nabla \vp)w\df x=\int_\Om (\nabla u\cdot\nabla\vp)w\df x=\lg \mcA u, \vp\rg 
		\end{split}
	\end{equation*}
	in the sense of distribution.  i.e., $\mcA u=\lim_{m\ra\iy}\mcA u_m$ strongly in $\mcD(\mcA^{\theta-1})$. And hence \eqref{09.03.7} holds. Finally, we have 
	\begin{equation*}
		\begin{split}
			\|\mcA u\|_{\mcD(\mcA^{\theta-1})}^2=\sum_{n=1}^\iy (u,\Phi_n)_{L^2(\Om)}^2\la_n^2\la_n^{2\theta-2}=\sum_{n=1}^\iy (u,\Phi_n)_{L^2(\Om)}^2\la_n^{2\theta}=\|u\|_{\mcD(\mcA^\theta)}^2. 
		\end{split}
	\end{equation*}

	It is obvious that $\mcA$ is a linear operator. 
\end{proof}

\begin{remark}\label{09.03.R1}
	For $\theta > \frac{1}{2}$, we still do not know whether Lemma \ref{09.03.L2} holds, since $\mathcal{D}(\mathcal{A}^{-\theta})$ may not be contained in $\ms{D}'(\Omega)$.
\end{remark}

\begin{lemma}\label{08.13.L3}
	For each  $0\leq \theta\in\R$ and $f\in \mcD(\mcA^\theta)$. Then, for each $0\neq\la\in\R$, the solution $z\in\mcD(\mcA)$ of 
	\begin{equation*}
		z+\la^2\mcA z=f
	\end{equation*}
	satisfies $z\in \mcD(\mcA^{\theta+1})$.
\end{lemma}

\begin{proof}
	Let $f=\sum_{n=1}^\iy f_n\Phi_n$ with $f_n=(f,\Phi_n)_{L^2(\Om)}$,  and $z=\sum_{n=1}^\iy z_n \Phi_n$ with $z_n=(z,\Phi_n)_{L^2(\Om)}$. Then, from   Lemma \ref{09.03.L2}, we have 
	\begin{equation*}
		\sum_{n=1}^\iy f_n\Phi_n=\sum_{n=1}^\iy z_n\Phi_n+\la^2\sum_{n=1}^\iy z_n\la_n\Phi_n=\sum_{n=1}^\iy z_n\left(1+\la^2\la_n\right)\Phi_n.
	\end{equation*}
	This implies that $f_n=z_n(1+\la^2\la_n)$ for all $n\in\N$. And hence
	\begin{equation*}
		\begin{split} 
		\iy>\sum_{n=1}^\iy f_n^2\la_n^{2\theta}
		&=\sum_{n=1}^\iy z_n^2(1+\la^2\la_n)^2\la_n^{2\theta}\\
		&=\sum_{n=1}^\iy z_n^2\left(\la^4\la_n^{2(\theta+1)}+2  \la^2\la_n^{1+2\theta}+ \la_n^{2\theta}\right)\geq \la^4\sum_{n=1}^\iy z_n^2\la_n^{2(\theta+1)}. 
		\end{split} 
	\end{equation*}
	Thus $z\in\mcD(\mcA^{\theta+1})$. 
\end{proof}

\begin{lemma} \label{08.13.L2}
	The operator
	\begin{equation*}
		\mf{A}:=
		\begin{pmatrix}
			0 & I\\
			-\mcA &0 
		\end{pmatrix}, \mbox{with domain}=\mcD(\mcA)\ts \mcD(\mcA^\f{1}{2})
	\end{equation*}
	generates a strongly continuous (s.c.) unitary group on $H_0^1(\Om;w)\ts L^2(\Om)$, where $\mcD(\mcA)$ and $\mcD(\mcA^\f{1}{2})$ is defined in Definition \ref{08.13.D1}. 
\end{lemma}

\begin{proof}
	 Take $E=H_0^1(\Om;w)\ts L^2(\Om)$. 
	
	{\it Step 1}.  It is clear that $\mf{A}$ is closed since $\mcA$ is a self-adjoint operator. 
	From (2) in Remark \ref{08.13.R1}, we know that $\mcD(\mcA)\ts\mcD(\mcA^\f{1}{2})$ is dense in $E$. 
	
	{\it Step 2}. For each $\bs f\in \mcD(\mcA^\f{1}{2})\ts L^2(\Om)$ and real $\la$ satisfying $0<|\la|<1$, the equation 
	\begin{equation*}
		\bs u-\la \mf{A}\bs u=f
	\end{equation*}
	has a unique solution $\bs u\in \mcD(\mcA)\ts \mcD(\mcA^\f{1}{2})$ and 
	\begin{equation}\label{08.14.2}
		\|\bs f\|_E^2\geq (1-|\la|)\|\bs u\|_E^2. 
	\end{equation}
	
	Indeed, let $\bs{f}=(f_1,f_2)\in \mcD(\mcA)\ts \mcD(\mcA^\f{1}{2})$.  Let $0\neq \la\in \R$, and let $w_1,w_2$ be the solutions of 
	\begin{equation*}
		w_i+\la^2\mcA w_i=f_i, \ i=1,2. 
	\end{equation*}
	Taking $u_1=w_1+\la w_2, u_2=w_2-\la \mcA w_1$, then from Lemma \ref{08.13.L3} we have $\bs{u}=(u_1,u_2)\in \mcD(\mcA^{\f{3}{2}})\ts \mcD(\mcA)$ is a solution of $\bs{u}-\la \mf{A}\bs{u}=\bs{f}$ with $\bs{f}=(f_1,f_2)$, and 
	\begin{equation*}
		\begin{split}
			u_1-\la u_2=f_1, \quad u_2+\la \mcA u_1=f_2. 
		\end{split}
	\end{equation*}
	From which, for all $0<|\la|<1$,  we have 
	\begin{equation*}
		\begin{split}
			\|\bs{f}\|_E^2
			&=\int_\Om |\nabla f_1|^2w\df x+\int_\Om f_2^2\df x=\int_\Om f_1\mcA f_1\df x+\int_\Om f_2^2\df x\\
			&=\int_\Om \left[(u_1-\la u_2)(\mcA u_1-\la \mcA u_2)+(u_2+\la \mcA u_1)^2\right]\df x\\
			&=\int_\Om|\nabla u_1|^2w\df x+\int_\Om u_2^2\df x+\la^2\int_\Om |\nabla u_2|^2w\df x+\la^2\int_\Om \left(\mcA u_1\right)^2\df x\\
			&\geq \|\bs{u}\|_E^2\geq (1-|\la|)^2 \|\bs{u}\|_E^2. 
		\end{split}
	\end{equation*}
	From (2) in Remark \ref{08.13.R1} we know that $\mcD(\mcA)\ts \mcD(\mcA^{\f{1}{2}})$ is dense in $\mcD(\mcA^\f{1}{2})\ts L^2(\Om)$, hence, for each $0<|\la|<1$, we have  \eqref{08.14.2}. 
	
	{\it Step 3}. 
	Note that $\bs u=(I-\la \mf{A})^{-1}\bs f$ for $0<|\la|<1$ from \eqref{08.14.2}, then, for each $|\mu|>1$, we have 
	\begin{equation*}
		\left\|(\mu I-\mf{A})^{-1}\right\|\leq \f{1}{|\mu|-1}. 
	\end{equation*}
	From Theorem 1.6.3 in \cite{Pazy} at pp.\! 23, it follows that $\mf{A}$ is the infinitesimal generator  of a group $\mcS(t)$ satisfying 
	\begin{equation*}
		\|\mcS(t)\|\leq e^{|t|}. 
	\end{equation*}
	
	{\it Step 4}. Return to the initial value problem \eqref{08.14.3}, 
	set 
	\begin{equation}\label{08.16.1}
		(y(x,t), z(x,t))=\mcS(t)(y_0(x),y_1(x)),
	\end{equation}
	then we have 
	\begin{equation*}
		\f{\pt}{\pt t}(y,z)=\mf{A}(y,z)=(z, -\mcA y), 
	\end{equation*}
	and hence $y$ is a solution of \eqref{08.14.3}. 
\end{proof}

\begin{remark}\label{08.14.R1}
	From Theorem \ref{09.02.T1} and Lemma \ref{08.13.L2}, we know that the mild solution of \eqref{08.14.3} and the weak solution of \eqref{08.14.3} are the same. 
\end{remark}

\subsection{Mild solution with non-homogeneous boundary condition}

We combine the lifting via the Dirichlet map from Subsection 2.2 with the semigroup/cosine formulation from Section 2.3 to obtain a mild solution representation for the boundary-input problem.

In this section, we shall establish the mild solution of the equation 
\begin{equation}\label{09.18.1}
	\begin{cases} 
		\pt_{tt}y-\Div(w\nabla y)=0, &\mbox{in }Q,\\
		y(0)=y_0, \pt_t y(0)=y_1, &\mbox{in }\Om,\\
		y=u, &\mbox{on }\Sigma, 
	\end{cases}
\end{equation}
where $u\in H^1(\Sigma)$. This follows from \cite{Lasiecka1,Lasiecka2}.

The operator 
\begin{equation*}
	\mcA: L^2(\Om)\sps \mcD(\mcA)\ra L^2(\Om)
\end{equation*}
generates a strongly continuous (s.c.) cosine operator $C(t)$ (see Lemma \ref{08.13.L2}, or Theorem 5.9 in \cite{Fattorini} at pp.\! 91) on $L^2(\Om)$ with 
\begin{equation*}
	S(t)z=\int_0^tC(\tau)z\df\tau, \ z\in L^2(\Om). 
\end{equation*}
Indeed, we have 
\begin{equation*}
	\mcS(t)=
	\begin{pmatrix}
		C(t) & S(t) \\
		\mcA S(t) &C(t)
	\end{pmatrix}, 
\end{equation*}
where $\mcS$ is defined in \eqref{08.16.1}. 
And the solution of \eqref{08.14.3} is 
\begin{equation}\label{08.15.7}
	y=C(t)y_0+S(t)y_1. 
\end{equation}
Moreover, we have 
\begin{equation}\label{08.19.1}
	\begin{split}
		C(t)y_0
		&=\sum_{n=1}^\iy \cos \la_n^\f{1}{2}t(y_0, \Phi_n)_{L^2(\Om)}\Phi_n, \ y_0\in H_0^1(\Om;w),\\
		S(t)y_1
		&=\sum_{n=1}^\iy \la_n^{-\f{1}{2}}\sin \la_n^\f{1}{2}t (y_1,\Phi_n)_{L^2(\Om)}\Phi_n, \ y_1\in L^2(\Om). 
	\end{split}
\end{equation}
It is easily verified that for all $\theta\geq 0$ and  $z\in \mcD(\mcA^\theta)$ we have 
\begin{equation}\label{08.18.3}
	C(t)z=\sum_{n=1}^\iy \cos\la_n^\f{1}{2} t(z, \Phi_n)_{L^2(\Om)}\Phi_n\in \mcD(\mcA^\theta).
\end{equation}
An important property is 
\begin{equation}\label{08.15.4}
	\begin{split}
		&S(t)L^2(\Om)\s \mcD(\mcA^\f{1}{2}), \mbox{and }\\
		&\mbox{the map}: t\mapsto \mcA^\f{1}{2}S(t) z \mbox{ is   continuous}, z\in L^2(\Om), \mbox{and }\\ &\|\mcA^\f{1}{2}S(t)\|_{\mcL(L^2(\Om))}\leq M_\ga e^{\ga t}, t\in\R, 
	\end{split}
\end{equation}
easily proved for $\mcA$ self-adjoint. Moreover, we have 
\begin{equation}\label{09.03.3}
	\f{\df }{\df t}C(t)z=-\mcA S(t)z \mbox{ for } z\in \mcD(\mcA^\f{1}{2}),\quad \f{\df^2}{\df t^2}C(t)z=-\mcA C(t)z \mbox{ for } z\in \mcD(\mcA), 
\end{equation}
and 
\begin{equation*}
	C(t)z-z=-\mcA\int_0^t\tau C(t-\tau)z\df \tau \mbox{ for } z\in L^2(\Om), 
\end{equation*}
and 
\begin{equation*}
	C(t)z-z=-\mcA\int_0^tS(t-\tau)z\df\tau=-\mcA\int_0^tS(\si)z\df\si \mbox{ for } z\in L^2(\Om). 
\end{equation*}

\subsection{Non-homogeneous Dirichlet boundary input and Dirichlet map}

We now consider the degenerate wave equation with non-homogeneous Dirichlet boundary input. Unlike the homogeneous Dirichlet case treated above, the boundary datum enters the evolution through an unbounded control action, and a direct reduction $y=z+Du$ would formally produce a forcing term involving time derivatives of $Du(t)$. Since our standing assumption on the boundary input is $u\in H^1(\Sigma)$, we do \emph{not} rely on second time derivatives of $Du(t)$. Instead, we follow the standard semigroup/cosine-operator approach for boundary control (see \cite{Lasiecka1,Lasiecka2}) and define the contribution of the boundary input through an admissible operator expressed in terms of the Dirichlet map $D$ and the sine family $S(t)$ generated by $\mcA$.

Recall that $D: H^{\f{1}{2}}(\Ga)\to H^1(\Om;w)$ is defined in Subsection 2.2 (see \eqref{08.14.5}).

From above, we obtain the solution of \eqref{07.08.01} in the following 

\begin{lemma}\label{08.19.L1}
	Under Assumption \ref{06.27.A1}, let $\mcA$ be the self-adjoint operator introduced above, and let $\mcA$ generate a cosine operator denoted by $C(t)$. Assume $u\in H^1(\Sigma)$. Then for each $y_0\in H_0^1(\Om;w)$ and $y_1\in L^2(\Om)$, the equation \eqref{07.08.01} has 
	\begin{equation}\label{08.15.6}
		y=C(t)y_0+S(t)y_1+(Lu)(t), 
	\end{equation}
	where 
	\begin{equation}\label{09.03.2}
		(Lu)(t)=\mcA\int_0^tS(t-\tau)Du(\tau)\df\tau. 
	\end{equation}
	The integral belongs to the natural energy scale associated with $\mcA$ and the action of $\mcA$ in \eqref{09.03.2} is understood in the extrapolation/dual space (equivalently in $V'$), as in the boundary-control framework of \cite{Lasiecka1,Lasiecka2}.
	Moreover, we have 
	\begin{equation}\label{09.03.4}
		\begin{split} 
			\pt_ty 
			&=\mcA S(t) y_0+C(t) y_1+\mcA \int_0^t C(t-\tau)Du(\tau)\df \tau,\\
			(L_tu)(t)
			&=\left(\f{\df L}{\df t}u\right)(t)=\mcA\int_0^t C(t-\tau)Du(\tau)\df\tau. 
		\end{split} 
	\end{equation}
\end{lemma}

\begin{proof}
	This representation is a standard consequence of the cosine-operator approach for Dirichlet boundary control; see (3.5) in \cite{Lasiecka1} at p.\! 172, or (a) at p.\! 39 and Theorem 3.1 at p.\! 42 in \cite{Lasiecka2}. It should be emphasized that the proof of Theorem 3.1 in \cite{Lasiecka2} makes use exclusively of semigroup theory. 
\end{proof}

\begin{corollary}\label{08.19.C2}
	Under conditions in Lemma \ref{08.19.L1}, we have 
	\begin{equation*}
		\begin{split}
			(Lu)(t)
			&=-\sum_{n=1}^\iy \left\{\la_n^{-\f{1}{2}}\int_0^t \sin \la_n^\f{1}{2}(t-\tau) \left(u(\tau), \f{\pt\Phi_n}{\pt\nu_\mcA}\right)_{L^2(\Ga)}\df\tau\right\} \Phi_n,\\
			(L_tu)(t)
			&=-\sum_{n=1}^\iy \left\{\int_0^t\cos \la_n^\f{1}{2}(t-\tau)\left(u(\tau), \f{\pt\Phi_n}{\pt \nu_\mcA}\right)_{L^2(\Ga)} \df\tau\right\}\Phi_n. 
		\end{split}
	\end{equation*}
\end{corollary}

\begin{proof}
	This is (5.8a) and (5.8b) in \cite{Lasiecka2} at pp.\! 51. We just verify $Lu$ in the following. 
	
	From Lemma \ref{07.08.L1'} we have $\Phi_n\in H^2(\mcO(\Ga; \be))$, by \eqref{08.19.1} and the definition of $Du(\tau)$,   then
	\begin{equation*}
		\begin{split}
			(Du(\tau), \Phi_n)_{L^2(\Om)}
			&=\la_n^{-1}\int_\Om [Du(\tau)]\mcA\Phi_n\df x=-\la_n^{-1}\int_\Ga \f{\pt \Phi_n}{\pt\nu_\mcA} u(\tau)\df S.
		\end{split}
	\end{equation*}
	Hence, we obtain 
	\begin{equation}\label{08.20.1}
		\begin{split} 
		(Lu)(t)
		&=\sum_{n=1}^\iy\la_n^{-\f{1}{2}} \mcA\left(\int_0^t \sin\la_n^\f{1}{2}(t-\tau)\left(Du(\tau),\Phi_n\right)_{L^2(\Om)}\df \tau\hspace{1mm}\Phi_n\right)\\
		&=\sum_{n=1}^\iy \la_n^\f{1}{2}\int_0^t\sin \la_n^\f{1}{2}(t-\tau)(Du(\tau), \Phi_n)_{L^2(\Om)}\df\tau\Phi_n\\
		&=-\sum_{n=1}^\iy \left\{\la_n^{-\f{1}{2}}\int_0^t \sin \la_n^\f{1}{2}(t-\tau) \left(u(\tau), \f{\pt\Phi_n}{\pt\nu_\mcA}\right)_{L^2(\Ga)}\df\tau\right\} \Phi_n.
		\end{split} 
	\end{equation}
	This completes the proof. 
\end{proof}

\begin{corollary}\label{09.03.C1}
	Under Assumption \ref{06.27.A1}. Let $u\in H^{1}(\Sigma)$. Then  
	\begin{equation*}
		Lu\in C([0,T]; H^1(\Om;w)),\quad L_tu\in C([0,T]; L^2(\Om)), 
	\end{equation*}
	where $Lu, L_tu$ are defined in \eqref{08.15.6} and \eqref{09.03.4} respectively. 
\end{corollary}

\begin{proof}
	Let $y_0=y_1=0$. Then we have $y=(Lu)(t)$ is the solution of the following equation
	\begin{equation*}
		\begin{cases}
			\pt_{tt}y+\mcA y=0, &\mbox{in }Q, \\
			y(0)=\pt_ty(0)=0, &\mbox{in }\Om,\\
			y=u, &\mbox{on }\Sigma. 
		\end{cases}
	\end{equation*}
	From Theorem \ref{09.02.T1}, we obtain 
	\begin{equation*}
		y\in C([0,T]; H^1(\Om;w))\cap C^1([0,T]; L^2(\Om)), 
	\end{equation*}
	and hence we have prove the corollary. 
\end{proof}

We emphasize that the representation \eqref{09.03.2}--\eqref{09.03.4} defines the control-to-state map $L$ for boundary inputs $u\in H^1(\Sigma)$ without requiring higher time regularity. In Section~3, this operator will be paired with solutions of the adjoint system through the boundary trace/conormal derivative, yielding the duality identity underlying the HUM criterion.

In summary, Section 2 has established the weak-solution framework and functional setting for the degenerate wave equation (Subsection 2.1). It then provides well-posedness and energy estimates for non-homogeneous Dirichlet boundary inputs via Theorem \ref{09.02.T1} and the bounds \eqref{09.04.1}--\eqref{09.04.2}. The semigroup/cosine formulation in Subsections 2.3 and 2.4 yields mild-solution representations compatible with these estimates. The Dirichlet map and lifting mechanism in Subsection 2.5 supply the bridge between boundary inputs and the homogeneous-boundary evolution framework. Sections 3 and 4 will build on these tools to develop the controllability analysis.

\section{Application: Approximate controllability}\label{S3}

Define $D^*$ as the adjoint operator of $D$ (see \eqref{08.14.5}) via the duality pairing
\begin{equation*}
	(Dv, z)_{L^2(\Om)}=\langle v, D^*z\rangle_{H^{1/2}(\Ga),\,H^{-1/2}(\Ga)}, 
\end{equation*}
then
\begin{equation*}
	D^*: L^2(\Om)\ra H^{-\f{1}{2}}(\Ga)
\end{equation*}
is a continuous operator, and  
\begin{equation*}
	-D^*\mcA z=\f{\pt z}{\pt\nu_\mcA}, \ z\in \mcD(\mcA)\cap H^2(\mcO(\Ga;\be)), 
\end{equation*}
in $H^{-1/2}(\Ga)$. By Theorem \ref{09.02.T1} we know that 
\begin{equation*}
	-D^*\mcA y=\f{\pt y}{\pt\nu_\mcA}
\end{equation*}
in $H^{-1/2}(\Ga)$ for each solution of \eqref{09.02.1}. 

Under Assumption \ref{06.27.A1}, for each $t>0$, denote  $E=H_0^1(\Om;w)\ts L^2(\Om)$, and 
\begin{equation*}
	K_t=\left\{
	\begin{pmatrix}
		y(t)\\
		\pt_ty(t)
	\end{pmatrix}\in E\colon y \mbox{ is the solution of  } \eqref{07.08.01} \mbox{ with } u\in H^1(\Sigma) \mbox{ and } y_0=y_1=0\right\}, 
\end{equation*}
and 
\begin{equation*}
	K=\bu_{t>0}K_t. 
\end{equation*}
It is clear that $K$ is the {\it attainable set} of equation \eqref{07.08.01}. 

\begin{definition}\label{09.14.D1}
	We say system \eqref{07.08.01} is {\it approximately controllable} if and only if $K$ is dense in $E$, i.e., $\overline{K}=E$. 
\end{definition}

\begin{lemma}\label{09.14.L1}
	$(z_0^*,z_1^*)\in K_t^\bot$ if and only if 
	\begin{equation*}
		\begin{split}
			D^*\mcA S(t-s)z_0^*+D^*\mcA C(t-s)z_1^*=0
		\end{split}
	\end{equation*}
	for all $s\in [0,t]$. 
\end{lemma}

\begin{proof}
	Since $(z_0^*,y(t))_{L^2(\Om)}+(z_1^*,\pt_ty(t))_{L^2(\Om)}=0$, and from \eqref{08.20.1} we obtain 
	\begin{equation*}
		\begin{split}
			(z_0^*,y(t))_{L^2(\Om)}
			&=(z_0^*,(Lu)(t))_{L^2(\Om)}=\sum_{n=1}^\iy \la_n^{\f{1}{2}}\int_0^t\sin \la_n^\f{1}{2}(t-\tau)(Du(\tau), \Phi_n)_{L^2(\Om)}(z_0^*, \Phi_n)_{L^2(\Om)}\df \tau\\
			&=\int_\Om\int_0^t\left(\mcA \sum_{n=1}^\iy \la_n^{-\f{1}{2}}\sin \la_n^\f{1}{2}(t-\tau)(z_0^*,\Phi_n)_{L^2(\Om)}\Phi_n\right)Du(\tau)\df x\df \tau \\
			&=\int_0^t\int_\Om \mcA S(t-\tau)z_0^* Du(\tau)\df x\df \tau=\int_0^t\left\langle D^*\mcA S(t-\tau)z_0^*,\, u(\tau)\right\rangle_{H^{-1/2},H^{1/2}}\df \tau, 
		\end{split}
	\end{equation*}
	and by the same argument as above we have 
	\begin{equation*}
		(z_1^*,\pt_ty(t))_{L^2(\Om)}=\int_0^t\left\langle D^*\mcA C(t-\tau)z_1^*,\, u(\tau)\right\rangle_{H^{-1/2},H^{1/2}}\df \tau, 
	\end{equation*}
	then we have proved the lemma since $H^1(\Sigma)$ is dense in $H^{1/2}(\Sigma)$. 
\end{proof}

\begin{theorem}\label{09.14.T1}
	The equation \eqref{07.08.01} is approximately controllable at time $T>0$ if and only if 
	\begin{equation*}
		\f{\pt z}{\pt \nu_\mcA}=0 \mbox{ on } \Sigma \ \Ra\ z_0=z_1=0 \mbox{ in }\Om, 
	\end{equation*}
	where $z$ is a solution of the following equation 
	\begin{equation}\label{09.14.1}
		\begin{cases}
			\pt_{tt}z+\mcA z=0, &\mbox{in }Q, \\
			z(T)=z_0, \pt_tz(T)=z_1, &\mbox{in }\Om,\\
			z=0, &\mbox{on }\Sigma, 
		\end{cases}
	\end{equation}
	and $z_0\in H_0^1(\Om;w), z_1\in L^2(\Om)$. 
\end{theorem}

\begin{proof}
	Note that $z(t)=C(T-t)z_0+S(T-t)z_1$ is the solution of \eqref{09.14.1}, then we have proved the theorem from Lemma \ref{09.14.L1} with $z_0=z_1^*$ and $z_1=z_0^*$. 
	This is also the HUM in \cite{Lions}. 
\end{proof}

\begin{remark}
Theorem \ref{09.14.T1} reduces approximate controllability to a unique continuation/observability property for the adjoint system. For degenerate operators, such properties can be delicate or may fail, so concrete verification is model-dependent and is not pursued here.
\end{remark}

\subsection*{Examples and open problems}
\phantomsection
\label{subsec:examples}

We briefly indicate how the abstract framework developed in Section~2 and the HUM-type criterion in Theorem~\ref{09.14.T1} apply to representative degenerate models. The main point is that once Assumption~\ref{06.27.A1} holds (so that the well-posedness and the lifting procedure via the Dirichlet map are available), the approximate controllability problem reduces to the corresponding unique continuation/observability property encoded in Theorem~\ref{09.14.T1}.

\paragraph{Example 3.1 (interior single-point degeneracy).}
Let $0\in\Omega\subset\mathbb{R}^N$ and consider the scalar weight
\[
w(x)=|x|^{\alpha}, \qquad \alpha\in(0,2),
\]
so that $\mathcal{A}y=-\Div(w\nabla y)$ degenerates only at the interior point $x=0$ and $w\ge \Lambda>0$ in a neighborhood of $\Gamma=\partial\Omega$.
For such weights (in the range above) one has $w\in A_2$ and the weighted Poincar\'e/compactness properties required in Assumption~\ref{06.27.A1}; see, e.g., the discussions in \cite{Yang1} and the references therein. Hence the well-posedness results of Section~2 and the boundary lifting $y=z+Du$ apply.
Consequently, the approximate controllability of \eqref{07.08.01} is characterized by Theorem~\ref{09.14.T1}, i.e., it is reduced to the vanishing of the conormal boundary observation $\partial_{\nu_{\mathcal A}}z$ for solutions of the adjoint system.

\paragraph{Grushin-type degeneracies and anisotropic models (discussion).}
Classical Grushin operators are typically anisotropic and can be written in divergence form with a degenerate symmetric matrix field $B(x)$, rather than a scalar weight $w(x)$.
Although the present paper is written for the scalar-weight prototype $\mathcal{A}=-\Div(w\nabla\cdot)$ in order to focus on the boundary input, lifting, and cosine-operator machinery, the HUM-type criterion in Theorem~\ref{09.14.T1} suggests that, whenever the corresponding well-posedness and trace/conormal estimates are available, approximate controllability is again equivalent to a suitable unique continuation/observability property for the adjoint equation.
Exact controllability and related observability inequalities for Grushin-type wave equations have been investigated in \cite{Yang}.

\paragraph*{Open problem: internal degeneracy and unique continuation.}
As already indicated above, the key obstruction for concrete degenerate models is the verification of the unique continuation (or equivalently observability) property appearing in Theorem~\ref{09.14.T1}.
In particular, the hyperbolic equation with an interior single-point degeneracy
\begin{equation}\label{eq:singlepoint-degenerate-model}
	\pt_{tt}y-\Div(|x|^\alpha\nabla y)=0 \quad \mbox{in } Q
\end{equation}
is closely related to the setting of \cite{Yang1}, but a full treatment of the unique continuation/observability issue in the present boundary-control framework remains an interesting topic for future work.

\section{Concluding remarks}
We extended the Lions–Lasiecka–Triggiani framework to degenerate hyperbolic equations with boundary inputs, proving approximate controllability criteria. Open problems remain, including the exact controllability of waves with internal point degeneracies.

\bibliographystyle{elsarticle-num}

\begin{thebibliography}{99}

\bibitem{Bellassoued}
M.~Bellassoued and M.~Yamamoto,
\newblock {\em Carleman estimates and applications to inverse problems for hyperbolic systems},
\newblock Springer, Tokyo, 2017.

\bibitem{Cannarsa1}
F.~Alabau-Boussouira, P.~Cannarsa, and G.~Leugering,
\newblock Control and stabilization of degenerate wave equations,
\newblock {\em SIAM J. Control Optim.}, 55 (2017), 2052--2087.

\bibitem{Chen1}
G.~Chen,
\newblock Control and stabilization for the wave equation in a bounded domain,
\newblock {\em SIAM J. Control Optim.}, 17 (1979), 66--81.

\bibitem{Chen2}
G.~Chen,
\newblock Control and stabilization for the wave equation in a bounded domain II,
\newblock {\em SIAM J. Control Optim.}, 19 (1981), 114--122.

\bibitem{Evans}
L.~C. Evans,
\newblock {\em Partial Differential Equations} (Second Edition),
\newblock AMS, Providence, 2010.

\bibitem{Fattorini}
H.~O. Fattorini,
\newblock Ordinary differential equations in linear topological spaces I,
\newblock {\em J. Differ. Equ.}, 5 (1968), 72--105.

\bibitem{Fattorini1}
H.~O. Fattorini,
\newblock Boundary control systems,
\newblock {\em SIAM J. Control Optim.}, 6 (1968), 349--385.

\bibitem{Fragnelli}
G.~Fragnelli, D.~Mugnai, and A.~Sbaï,
\newblock Boundary controllability for degenerate/singular hyperbolic equations in nondivergence form with drift,
\newblock {\em Appl. Math. Optim.}, 91 (2025), 42--74.

\bibitem{GC}
J.~Garcia-Cuerva and J.~L. Rubio de Francia,
\newblock {\em Weighted Norm Inequalities and Related Topics},
\newblock North-Holland Publishing Co., Amsterdam, 1985.

\bibitem{Gueye}
M.~Gueye,
\newblock Exact boundary controllability of 1-D parabolic and hyperbolic degenerate equations,
\newblock {\em SIAM J. Control Optim.}, 52 (2014), 2037--2054.

\bibitem{Heinonen}
J.~Heinonen, T.~Kilpeläinen, and O.~Martio,
\newblock {\em Nonlinear Potential Theory of Degenerate Elliptic Equations},
\newblock Oxford University Press, New York, 1993.

\bibitem{Lasiecka}
I.~Lasiecka and R.~Triggiani,
\newblock Riccati equations for hyperbolic partial differential equations with $L^2(0,T;L^2(\Gamma))$ Dirichlet boundary terms,
\newblock {\em SIAM J. Control Optim.}, 24 (1986), 884--925.

\bibitem{Lasiecka1}
I.~Lasiecka, J.-L. Lions, and R.~Triggiani,
\newblock Non-homogeneous boundary value problems for second order hyperbolic operators,
\newblock {\em J. Math. Pures Appl.}, 65 (1986), 149--192.

\bibitem{Lasiecka2}
I.~Lasiecka and R.~Triggiani,
\newblock A cosine operator approach to modeling $L^2(0,T;L^2(\Gamma))$ boundary input hyperbolic equations,
\newblock {\em Appl. Math. Optim.}, 7 (1981), 35--93.

\bibitem{Lasiecka3}
I.~Lasiecka, R.~Triggiani, and P.-F. Yao,
\newblock An observability estimate in $L^2(\Omega)\times H^{-1}(\Omega)$ for second-order hyperbolic equations with variable coefficients,
\newblock in {\em Control of Distributed Parameter and Stochastic Systems}, IFIP, 1999, pp.~71--78.

\bibitem{Lions}
J.-L. Lions and E.~Magenes,
\newblock {\em Non-homogeneous Boundary Value Problems and Applications},
\newblock Springer-Verlag, Berlin, 1972.

\bibitem{Pazy}
A.~Pazy,
\newblock {\em Semigroups of Linear Operators and Applications to Partial Differential Equations},
\newblock Springer, New York, 1983.

\bibitem{Slemrod}
M.~Slemrod,
\newblock Stabilization of boundary control systems,
\newblock {\em J. Differ. Equ.}, 22 (1976), 402--415.

\bibitem{Trudinger}
N.~S. Trudinger,
\newblock Linear elliptic operators with measurable coefficients,
\newblock {\em Ann. Scuola Norm. Sup. Pisa}, 27 (1973), 265--308.

\bibitem{Yang}
D.~Yang, W.~Wu, B.-Z. Guo, and S.~Chai,
\newblock On exact controllability for a class of 2-D Grushin hyperbolic equations,
\newblock {\em J. Differ. Equ.}, 448 (2025), 113710.

\bibitem{Yang1}
D.~Yang, W.~Wu, and B.-Z. Guo,
\newblock On controllability of a class of $N$-dimensional hyperbolic equations with internal single-point degeneracy,
\newblock preprint, 2025.

\bibitem{Yao}
P.-F. Yao,
\newblock {\em Modeling and Control in Vibrational and Structural Dynamics: A Differential Geometric Approach},
\newblock CRC Press, Boca Raton, 2014.

\bibitem{Gao}
M.~Zhang and H.~Gao,
\newblock Interior controllability of semi-linear degenerate wave equations,
\newblock {\em J. Math. Anal. Appl.}, 457 (2018), 10--22.

\bibitem{Zuazua}
E.~Zuazua,
\newblock {\em Exact Controllability and Stabilization of the Wave Equation},
\newblock Springer, New York, 2024.

\end{thebibliography}

\end{document}